\documentclass[12pt, notitlepage]{article}
\usepackage{amssymb,amsmath,comment}
\usepackage{amsthm}
\allowdisplaybreaks
\catcode`\@=11 \@addtoreset{equation}{section}
\def\thesection{\arabic{section}}

\def\theequation{\thesection.\arabic{equation}}
\catcode`\@=12
\usepackage{colortbl}
\usepackage{hyperref}
\usepackage[mathscr]{eucal}
\usepackage{epsf}
\usepackage{esint}
\usepackage{a4wide}
\def\R{\mathbb{R}}

\newcommand{\noi} {\noindent}

\usepackage[all]{xy}
\catcode`\@=11
\def\theequation{\@arabic{\c@section}.\@arabic{\c@equation}}
\catcode`\@=12

\def\N{{I\!\!N}}

\newtheorem{theorem}{Theorem}[section]

\newtheorem{lem}[theorem]{Lemma}
\newtheorem{prop}[theorem]{Proposition}

\newtheorem{rmk}[theorem]{Remark}
\newtheorem{defn}{Definition}[section]

\newtheorem{notn}{Notation}[section]

\begin{document}

{\vspace{0.01in}}

\title{Multiplicity of solutions to a class of degenerate elliptic equations in both sub-critical and critical cases}
\author{Kaushik Bal\footnote{Department of Mathematics and Statistics, Indian Institute of Technology Kanpur, Kanpur-208016, Uttar Pradesh, India, Email: kaushik@iitk.ac.in} and Sanjit Biswas\footnote{Department of Mathematics and Statistics, Indian Institute of Technology Kanpur, Kanpur-208016, Uttar Pradesh, India, Email: sanjitbiswas410@gmail.com}}

\maketitle

\begin{abstract}\noindent
Given a smooth, bounded domain $\Omega\subset\R^N$, we establish the existence of two non-trivial, non-negative solutions to the semilinear degenerate elliptic equation
\begin{align*}
    \left. \begin{array}{l}
        -\Delta_\lambda u=\mu g(z)|u|^{r-1}u+h(z)|u|^{s-1}u \;\text{in}\; \Omega \\
         u\in H^{1,\lambda}_0(\Omega)
    \end{array}\right\}
    \end{align*}
where $\Delta_\lambda=\Delta_x+|x|^{2\lambda}\Delta_y$ denotes the Grushin Laplacian Operator, $z=(x,y)\in\Omega$, $N=n+m;\, n,\, m\geq 1$, $\lambda>0$, $0\leq r<1<s<2^*_\lambda-1$ and $\mu$ is a positive parameter. The functions $g$ and $h$ may change sign and $2^*_\lambda=\frac{2Q}{Q-2}$ is the critical Sobolev exponent associated with the homogeneous dimension $Q=n+(1+\lambda)m$ of $\Delta_\lambda$. In the critical case $s=2^*_\lambda-1$, we further show that the problem admits at least two non-trivial, non-negative solutions under the additional assumptions $g\geq 0$ and $h\equiv 1$.
\end{abstract}

\noi {Keywords: Grushin operator, convex-concave nonlinearity, Nehari manifold technique, critical nonlinearity, Mountain pass theorem.}

\noi{\textit{2020 Mathematics Subject Classification: 35A15, 35J70, 35R01}

\bigskip

\tableofcontents

\section{Introduction and main results }
This article investigates the multiplicity of solutions to a class of semilinear degenerate elliptic equations. In particular, we study 
\begin{align}\label{ME}
    \left. \begin{array}{l}
        -\Delta_\lambda u=\mu g(z)|u|^{r-1}u+h(z)|u|^{s-1}u \;\text{in}\; \Omega \\
         u\in H^{1,\lambda}_0(\Omega)
    \end{array}\right\} 
\end{align}
where $\Delta_\lambda u=\Delta_x u+|x|^{2\lambda}\Delta_y u$ is the Grushin Laplacian, $z=(x,y)\in\Omega$, with $N=n+m;\, n,\, m\geq 1,\,\lambda>0$, $0\leq r<1<s<2^*_\lambda-1$ and $\mu$ is a positive parameter. The bounded domain $\Omega\subset \R^{N}$ is smooth, and the weight functions $g,\,h$ may change sign. When $\lambda=0$, the operator reduces to the usual Laplacian operator. This operator was first introduced by Baouendi in \cite{Baou} and later studied by Grushin \cite{Grushin70, Grushin71}. The hypoellipticity properties (see \cite{Hor}) of $\Delta_\lambda$ were analyzed by Grushin in \cite{Grushin70}. This operator arises naturally in the context of partial differential equations on manifolds. In \cite[Note 2]{KOG}, Kagoj-Lanconelli demonstrated that the Grushin operator also appears in the analysis of partial differential equations in the Heisenberg group. In addition, some connections between this operator and the PDEs in hyperbolic spaces have been explored in \cite{sandeep}. Since the weight $|x|^{2\lambda}$ appears in the definition of the Grushin operator, the operator fails to be uniformly elliptic on domains intersecting the plane $\{x=0\}$. We always assume that $$\Omega\cap\{z=(x,y)\in\R^N: x=0\}\neq\emptyset.$$ 

In \cite{FL}, Franchi-Lanconelli obtained several regularity and embedding results for a class of operators, including Grushin. Monticelli \cite{Monti} established a maximum principle, Hopf's lemma, and qualitative properties of solutions for degenerate equations, including the Grushin operator. Recently, a Berestycki-Lions type result has been proved by Alves et al. in \cite{Alves23}. Very recently, Alves et al. \cite{PL} investigated the existence result of the Brezis-Nirenberg type problem for this setting. More precisely, when $r\in[1,2^*_\lambda)$, they obtained the existence of a weak solution to \textcolor{black}{problem \eqref{ME}} with $s=2^*_\lambda-1$ under various conditions on the dimension and the parameter. Nevertheless, the equation has not yet been studied for the case $r\in[0,1)$, even for the subcritical case. This article addresses these gaps and discusses several multiplicity results for small parameter values $\mu>0$. For the Grushin operator, however, multiplicity phenomena of convex-concave type remain comparatively less understood, although several existence results may be found in \cite{KS, KOG}, and references therein. It should also be noted that for $n\geq 2$, no suitable maximum principle is available due to the effect of degeneracy, so we focus exclusively on nonnegative solutions. However, when $n=1$, a partial maximum principle has been established in \cite{KS}. For more problems related to the Grushin operator, we refer to \cite{WB01, Ambrosio04, Duong17, Duong, RM} and the references therein.

The motivation behind our study traces back to the classical work of  Ambrosetti, Brezis, and Cerami \cite{BR}, who considered the problem
\begin{align}\label{UE}
    \left. \begin{array}{l}
        -\Delta u=\mu g(x)u^r+h(x)u^s\; \text{in}\;\Omega\\
             u\geq 0\text{ in }\Omega \text{ and } u=0\;\text{on}\;\partial\Omega
    \end{array}\right\} 
\end{align}
with $g=h=1$ and proved the existence of $\mu_0>0$ such that equation (\ref{UE}) admits at least two positive solutions for every $\mu\in(0,\mu_0)$ when $0<r<1<s\leq 2^*-1$ among other results. For the sign-changing functions $g$ and $h$, De Figueiredo et al. \cite{DE} obtained multiple solutions by combining a mountain pass construction with local minimization.

\textcolor{black}{In the borderline case where $r=1$ and $\int h\phi_1^{s+1}<0$, Brown and Zhang \cite{BZ} proved the existence of multiple solutions for $\mu_1<\mu<\mu_1+\delta$, for some $\delta>0$. Here $(\phi_1,\mu_1)$ is the principal eigenpair of the Dirichlet boundary value problem $-\Delta u=\mu g(x)u$. Their approach utilized the Nehari manifold method in the subcritical setting. Wu \cite{WU1, WU} further investigated the problem and established the existence of at least two solutions to problem (\ref{UE}) with a sign-changing weight function $g$ and $h\equiv 1$ when $0<r<1<s<2^*-1$ and later under more general hypotheses on $g$ and $h$ for small $\mu>0.$}

Related results for quasilinear operators can be found in \cite{Adi1, BN, Pavel} and the references therein. Interested readers may also find \cite{ Michel, DE} useful, which deals with the impact of the concave and convex nonlinearities on the number of positive solutions.
We now mention structural assumptions on the weight functions $g$ and $h$.
\begin{enumerate}
    \item[($g_1$)] There exists $q\in(r+1, 2^*_\lambda)$ such that $g\in L^a(\Omega)$ where $a=\frac{q}{q-(r+1)}$ and $g^+$ is not identically zero function in $\Omega$. In case $r=0$, we also assume $g$ is non-negative in $\Omega$.
    \item[($g_2$)] There exists a $\rho-$ball (see \ref{gaugeb}) $B_{2R}(z_0)\Subset\Omega\setminus\Sigma$ such that $g\in L^\infty(B_R(z_0)),$ where $\Sigma =\{(x,y)\in\R^N| x=0\}.$
    \item[($h_1$)] There exists $p\in(s+1, 2^*_\lambda)$ such that $h\in L^b(\Omega)$ where $b=\frac{p}{p-(s+1)}$ and $h^+$ is not identically zero function in $\Omega$.
\end{enumerate}
Under these hypotheses, our main result reads as follows:
\begin{theorem}\label{MT1}(Multiplicity in the sub-critical regime).
    Assume $0\leq r<1<s<2^*_\lambda-1$ and that $(g_1)$, $(h_1)$ hold. Then there exists $\mu^*>0$ such that for every $\mu\in(0,\mu^*)$, the problem (\ref{ME}) possesses at least two non-trivial, non-negative weak solutions. Moreover, if $n=1,\lambda\geq 1$ and $g,h$ are non-negative, then these solutions are strictly positive.
\end{theorem}
 Motivated by \cite{WU1, WU}, we derive the existence of two solutions to problem (\ref{ME}), using the Nehari manifold technique. To this concern, the Nehari manifold is decomposed into three disjoint parts; one of these is empty for small parameter values, while each of the remaining two contains a critical point of the associated energy functional.
 
For the critical case, we study the equation
\begin{align}\label{ME1}
    \left. \begin{array}{l}
        -\Delta_\lambda u=\mu g|u|^{r-1}u+|u|^{2^*_\lambda-2}u \;\text{in}\; \Omega \\
         u\in H^{1,\lambda}_0(\Omega)
    \end{array}\right\} \tag{$E_\mu$}
\end{align}
 and obtain the following result:
\begin{theorem}\label{MT2}(Critical convex-concave case).
 Let $g\geq 0$ satisfy $(g_1)$ and $(g_2)$. Then there exists $\mu^*>0$ such that for every $\mu\in(0,\mu^*)$, the \textcolor{black}{problem (\ref{ME1})} with $0\leq r<1$ admits at least two non-trivial, non-negative weak solutions. Moreover, if $n=1\;,\lambda\geq 1$ then the solutions are strictly positive.     
\end{theorem}
 In order to prove Theorem \ref{MT2}, we mainly adopt the methods employed in \cite{FD, Silva}. We consider the associated $C^1-$ functional $I_\mu: H^{1,\lambda}_0(\Omega)\to\R$, 
 \begin{align}\label{CF}
     I_\mu(u)=\frac{1}{2}\int_\Omega |\nabla_\lambda u|^2\ dz-\frac{\mu}{r+1}\int_\Omega g\cdot (u^+)^{r+1}\ dz - \frac{1}{s+1}\int_\Omega (u^+)^{s+1}\ dz,
 \end{align}
and observe that every critical point of $I_\mu$ is non-negative, where $\nabla_\lambda u=(\nabla_x u, |x|^\lambda\nabla_y u)$. We show that the functional $I_\mu$ attains its local infimum in a neighborhood of the origin, yielding the first solution, while a second solution arises from the Mountain Pass Theorem. However, the lack of compactness of Palis-Smale sequences makes the analysis more delicate and prevents the direct application of the Mountain Pass Theorem. To overcome this difficulty, we prove that the mountain pass level is strictly less than a threshold $c^*$ such that the Palais-Smale condition is satisfied for every level strictly smaller than $c^*$ (see Lemma \ref{BrezisNiren1} and Lemma \ref{MPL}).

\begin{notn}
Throughout this paper, we will make use of the following notations.
\begin{itemize}
 \item $\Omega\subset\R^N$ is a bounded smooth domain, where $N=n+m;\, n, \, m\geq 1.$ 

 \item \textcolor{black}{By the notation $\omega\Subset\Omega$, we mean that $\omega$ is an open subset of $\Omega$ such that $\overline{\omega}\subset\Omega$.}

 \item $z=(x,y)$ denotes any arbitrary point in $\Omega$ and $Q=n+(1+\lambda)m$.

   \item  $\Sigma:=\{(x,\;y)\in\R^N|x=0\}$ is the degeneracy set for the Grushin operator.
    \item For $1\leq p\leq\infty$, the notation $\|u\|_p$ denotes the $L^p(\Omega)$ norm of the measurable function $u$, i.e.,$$\|u\|_p=\left (\int_\Omega |u|^p\, dz\right )^\frac{1}{p}.$$
    \item For a function $u$, we define $u^+:=\max{\{u,0\}}$ and $u^-=\max{\{-u, 0\}}$.
    \item $C,\, C_1,\, C_2,\;...$ denote positive constants, whose value may change from line to line.
    \item For $p>1$, $p'$ denotes the H\"{o}lder conjugate of $p$, defined as $p'=\frac{p}{p-1}$.
\end{itemize}
\end{notn}

\section{Preliminaries and variational framework}\label{Pri}
Let $\Omega\subset\R^{N(=n+m)}$ be a smooth bounded domain. The Grushin gradient of a function $u$ is denoted by $\nabla_\lambda u$ and defined as $\nabla_\lambda u=(\nabla_x u,\;|x|^\lambda \nabla_y u)$, where $\nabla_x u$ and $\nabla_y u$ are the usual gradients of the function $u$ with respect to $x$ and $y$ variables, respectively. Now, we define the Hilbert space associated to $\Delta_\lambda$ that will be used in the sequel  $$H^{1,\lambda}(\Omega):=\{ u\in L^2(\Omega): |\nabla_\lambda u|\in L^2(\Omega)\}$$
equipped with the norm
$$\|u\|=\left (\int_\Omega \left (u^2(z)+ |\nabla_\lambda u(z)|^2\right ) dz\right )^\frac{1}{2}.$$
 We also define the space $H^{1,\lambda}_0(\Omega)$ as closure of $C_c^\infty(\Omega)$ in $H^{1,\lambda}(\Omega)$. Due to \cite[Theorem 2.1]{MP09}, in the space $H^{1,\lambda}_0(\Omega)$, the above norm is equivalent to $$\|u\|_\lambda=\left (\int_\Omega |\nabla_\lambda u|^2 dz \right )^\frac{1}{2}.$$ 
 The gauge norm associated to the Grushin operator is denoted by $\rho$ and defined as 
 \begin{align}\label{gauge}
     \rho(z)=\left ( |x|^{2(1+\lambda)}+(1+\lambda)^2|y|^2\right )^\frac{1}{2(1+\lambda)},
 \end{align}
 which is homogeneous of degree one concerning the dilation $\delta_t(z)=(tx, t^{1+\lambda}y)$, where $z=(x,y)$ and $t>0$. The $\rho-$ball with radius $d>0$ and center at $z_1\in\R^N$ is denoted by $B_d(z_1)$ and defined as 
 \begin{align}\label{gaugeb}
     B_d(z_1):=\{z\in\R^N| \rho(z-z_1)<d\}.
 \end{align}
For more information about the Grushin operator and the Sobolev inequality associated with $\Delta_\lambda$, we refer to \cite{FL, KOG} and the references therein. We now state an embedding result, which can be found in \cite{KOG}.
\begin{lem}\label{Emb}
    Let $\Omega\subset\R^N$ be a smooth bounded domain. Then the embedding
    $H^{1,\lambda}_0(\Omega)\hookrightarrow L^p(\Omega)$ is continuous for every $p\in [1, 2^*_\lambda]$. Moreover, the above embedding is compact for all $p\in [1, 2^*_\lambda)$.
\end{lem}
We denote $S_p$ as the best constant in the embedding $H^{1,\lambda}_0(\Omega)\hookrightarrow L^p(\Omega)$ and
define \textcolor{black}{
\begin{align*}
    \mathcal{S}_\lambda=\inf\Big\{\int_\Omega |\nabla_\lambda u|^2\, dz: u\in H^{1,\lambda}_0(\Omega), \|u\|_{L^{2^*_\lambda}(\Omega )}=1\Big\},
\end{align*}}
which is independent of $\Omega$ (see \cite[Proposition 3.3]{PL}) and \textcolor{black}{is achieved in the} case of $\Omega=\R^{n+m}$. That is, there exists a non-negative $v\in H^{1,\lambda}_0(\Omega)$ satisfying the following equation:
\begin{align}\label{Tal}
    -\Delta_\lambda v=v^\frac{Q+2}{Q-2} \text{ in } \R^N,
\end{align}
where $Q=n+(1+\lambda)m$ is the homogeneous dimension of $\Delta_\lambda$.  For $\epsilon>0$, let us define $$v_\epsilon=\epsilon^\frac{2-Q}{2}v(\delta_{\frac{1}{\epsilon}}(z)),$$ which is also a solution to the above equation (\ref{Tal}). Let us consider the $\rho-$ball $B_{2R}(z_0)\Subset\Omega\setminus\Sigma$, which is introduced in the condition ($g_2$). Choose a function $\phi\in C^\infty_c(B_{2R}(z_0))$ such that $0\leq \phi\leq 1$ in $B_{2R}(z_0)$ and $\phi=1$ in $B_R(z_0)$. Now, define
\begin{align}\label{Asy}
    u_\epsilon(z)=\phi(z)v_\epsilon(z) \text{ and } w_\epsilon=\frac{u_\epsilon}{\|u_\epsilon\|_{\textcolor{black}{L^{2^*_\lambda}(\Omega)}}}.
\end{align}
To proceed, let us discuss a crucial asymptotic result that will play a significant role in establishing Theorem \ref{MT2}. The result is outlined as follows (see \cite{PL})
\begin{prop}\label{Pro1}
    As $\epsilon\to 0$, the following conclusions hold:
    \begin{enumerate}
        \item[(i)]$\int_\Omega |\nabla_\lambda u_\epsilon|^2\, dz\leq \mathcal{S}_\lambda^\frac{Q}{2}+O(\epsilon^{Q-2}),$
        \item [(ii)]$\int_\Omega |u_\epsilon|^{2^*_\lambda}\, dz=\mathcal{S}_\lambda^\frac{Q}{2}+O(\epsilon^Q)$,
        \item[(iii)] $\int_\Omega |u_\epsilon|^2 dz\geq\begin{cases}
            C\epsilon^2+O(\epsilon^{Q-2}) \text{ if } Q>4\\
          C\epsilon^2|\mathrm{ln}\epsilon|+O(\epsilon^2) \text{ if } Q=4\\
          C\epsilon^{Q-2}+O(\epsilon^2) \text{ if } Q<4.
        \end{cases}$
        \item[(iv)]For $\frac{2^*_\lambda}{2}<\gamma<2^*_\lambda$, we have
       $$\int_\Omega |u_\epsilon|^\gamma\, dz=O(\epsilon^{Q-\frac{\gamma(Q-2)}{2}}).$$
    \end{enumerate}
\end{prop}
\begin{proof}
     We refer to \cite[ Proposition 4.2 ]{PL} for the first three conclusions. Since $\frac{2^*_\lambda}{2}<\gamma<2^*_\lambda$, we have $0<Q-\gamma\left(\frac{Q-2}{2}\right )$ and $v\in L^\gamma(\R^N).$ Now, using the definition of $u_\epsilon$, we have
\begin{align*}
\int_\Omega|u_\epsilon|^\gamma\, dz=\int_\Omega|\phi v_\epsilon|^\gamma\, dz&=\int_{B_r}|v_\epsilon|^\gamma\, dz+\int_{B_{2r}\setminus{B_r}}|\phi v_\epsilon|^\gamma\, dz\nonumber\\
&\leq \epsilon^\frac{\gamma(2-Q)}{2}\left (\int_{B_r}|v(\delta_\frac{1}{\epsilon} (z))|^\gamma\,  dz+\|\phi\|_\infty\int_{B_{2r}\setminus{B_r}} |v(\delta_\frac{1}{\epsilon}(z))|^\gamma\, dz\right )\nonumber\\
&\leq \epsilon^\frac{\gamma(2-Q)}{2}\left (\int_{B_\frac{r}{\epsilon}}|v(\xi)|^\gamma \epsilon^Q d\xi+ \|\phi\|_\infty\int_{B_\frac{2r}{\epsilon}\setminus{B_\frac{r}{\epsilon}}} |v(\xi)|^\gamma\epsilon^Q d\xi\right )\nonumber\\
&\leq C\epsilon^{Q-\frac{\gamma(Q-2)}{2}}\int_{\R^N}|v(\xi)|^\gamma\, d\xi\leq M\epsilon^{Q-\frac{\gamma(Q-2)}{2}}.
\end{align*}
Thus, the proof is complete.
\end{proof}
Now we give the notion of a weak solution to \textcolor{black}{problem \eqref{ME}}.
\begin{defn}\label{def1}
    A function $u\in H^{1,\lambda}_0(\Omega)$ is said to be a 'weak solution' of  problem (\ref{ME}) if for every $\phi\in C_c^\infty(\Omega)$ the following identity holds: 
    $$\int_\Omega \nabla_\lambda u\cdot \nabla_\lambda \phi=\mu\int_\Omega g|u|^{r-1}u\phi dz + \int_\Omega h|u|^{s-1}u\phi dz.$$
\end{defn}
The  functional associated to the given problem (\ref{ME}) is denoted by $I_\mu: H^{1,\lambda}_0(\Omega)\to \R$ and defined as 
 $$I_\mu(u)=\frac{1}{2}\int_\Omega |\nabla_\lambda u|^2 dz-\frac{\mu}{r+1}\int_\Omega g|u|^{r+1} dz-\frac{1}{s+1}\int_\Omega h|u|^{s+1} dz,$$
which is a $C^1$-functional. In order to find solutions to the problem (\ref{ME}), we study the existence of critical points of $I_\mu$ via the Nehari manifold technique. The Nehari manifold corresponding to the functional $I_\mu$ is denoted by $\mathcal{M}_\mu$ and defined as
\begin{align}\label{N1}
    \mathscr{M}_\mu:=\{ u\in H^{1,\lambda}_0(\Omega)\setminus\textcolor{black}{\{0\}}: \langle I'_\mu(u),u\rangle=0 \}.
\end{align}
 For simplicity, let us define another functional $\tau: H^{1,\lambda}_0(\Omega)\to \R$ by
\begin{align}\label{N2}
    \tau(u)=\langle I'_\mu(u),u \rangle=\int_\Omega |\nabla_\lambda u|^2 dz-\mu\int_\Omega g|u|^{r+1} dz-\int_\Omega h|u|^{s+1} dz.
\end{align}
Therefore, $$\langle \tau'(u),u\rangle=2\int_\Omega |\nabla_\lambda u|^2 dz-\mu(r+1)\int_\Omega g|u|^{r+1} dz-(s+1)\int_\Omega h|u|^{s+1} dz.$$
We split the Nehari Manifold into three parts, $\mathscr{M}_\mu^+$, $\mathscr{M}_\mu^-$ and $\mathscr{M}_\mu^0$ such that 
$$\mathscr{M}_\mu=\mathscr{M}_\mu^+\cup\mathscr{M}_\mu^0\cup\mathscr{M}_\mu^-,$$
 where 
 \begin{align*}
     \mathscr{M}^+_\mu=\{ u\in\mathscr{M}: &\langle \tau'(u),u\rangle>0\},\, \mathscr{M}^-_\mu=\{ u\in\mathscr{M}: \langle \tau'(u),u\rangle<0\},\\
     &\mathscr{M}^0_\mu=\{ u\in\mathscr{M}: \langle \tau'(u),u\rangle=0\}.
 \end{align*}
We also define $$m_\mu:=\inf_{u\in\mathscr{M}_\mu} I_\mu(u),\ m^-_\mu:=\inf_{u\in\mathscr{M}^-_\mu}  I_\mu(u)\text{ and } m^+_\mu:=\inf_{u\in\mathscr{M}^+_\mu} I_\mu(u).$$
 
\section{Auxiliary results}\label{AUX}
This section discusses several auxiliary results that play a crucial role in proving our multiplicity result for the subcritical case. Firstly, we show that $\mathscr{M}^+_\mu$ and $\mathscr{M}^-_\mu$ are nonempty. To this aim, we analyze the fibering maps. For $u\in H^{1,\lambda}_0(\Omega)$, we consider the fibering map $F:[0,\infty)\to\R$,
$$F(t)=I_\mu(tu)=\frac{t^2}{2}\int_\Omega |\nabla_\lambda u|^2\, dz-\frac{\mu t^{r+1}}{r+1}\int_\Omega g|u|^{r+1}\, dz-\frac{t^{s+1}}{s+1}\int_\Omega h|u|^{s+1} dz.$$ Therefore,
\begin{align}\label{NM0}
   F'(t)=\frac{1}{t}\langle I'_\mu(tu),tu\rangle&=\frac{1}{t}\tau(tu)=t\int_\Omega |\nabla_\lambda u|^2\ dz-\mu t^{r}\int_\Omega g|u|^{1+r}\ dz-t^{s}\int_\Omega h|u|^{1+s}\ dz\nonumber\\
    &=t^{r}\left [\left ( t^{1-r}\int_\Omega |\nabla_\lambda u|^2\ dz-t^{s-r}\int_\Omega h|u|^{1+s}\ dz\right )-\mu\int_\Omega g|u|^{1+r} \ dz\right ].
\end{align}
\textcolor{black}{Now we define a function $G:[0,\infty)\to\R$ by
$$G(t)=t^{1-r}\int_\Omega |\nabla_\lambda u|^2\ dz-t^{s-r}\int_\Omega h|u|^{1+s}\ dz.$$
Therefore, we have}
\begin{align}\label{NM1}
    F''(t)=\frac{1}{t^2}\langle \tau'(tu),tu\rangle-\frac{1}{t^2}\tau(tu)=t^rG'(t)+rt^{r-1}\left (G(t)-\mu\int_\Omega g|u|^{1+r} \ dz\right ).
\end{align}
The function $G$ enjoys the following property.
\begin{lem}\label{LM1}
    Let $u\in H^{1,\lambda}_0(\Omega)$ \textcolor{black}{be} such that $\int_\Omega h|u|^{s+1}\;dz>0$. Then there exists $t_0>0$ such that $G$ is increasing on $[0,t_0]$ and decreasing on $[t_0,\infty)$\textcolor{black}{.}
\end{lem}
\begin{proof}
We have
\begin{align*}
       G'(t)=(1-r)t^{-r}\int_\Omega |\nabla_\lambda u|^2 dz-(s-r)t^{s-r-1}\int_\Omega h|u|^{s+1}\ dz\ \text{ for all } t>0.
   \end{align*}
It is easy to observe that $t_0=\left (\frac{(1-r)\int_\Omega |\nabla_\lambda u|^2\ dz}{(s-r)\int_\Omega h|u|^{s+1}\ dz}\right )^\frac{1}{(s-1)}$ is the only critical point of $G$ and $G'(t)>0$ in $(0,t_0)$ and $G'(t)<0$ in $[t_0,\infty)$. Furthermore, 
\begin{align}\label{SV}
    \sup_{t\geq 0}\, G(t)=G(t_0)=\left(\frac{1-r}{s-r}\right)^\frac{1-r}{s-1}\left(\frac{s-1}{s-r}\right )\frac{\|u\|_\lambda^{2(s-r)}}{\left(\int_\Omega h|u|^{1+s}\;dz\right)^\frac{1-r}{s-1}}.
\end{align}
This proves the Lemma.
\end{proof}
So far, we do not know \textcolor{black}{whether} the Nehari manifold $\mathscr{M}_\mu$ is nonempty. The following lemma ensures that $\mathscr{M}_\mu\neq\emptyset$.
\begin{lem}\label{Char}
   We assume that ($g_1$) and $(h_1)$ hold. Then there exists $\mu_0>0$ such that for every $\mu\in(0,\mu_0)$ and $u\in H^{1,\lambda}_0(\Omega)$ with $\int_\Omega h|u|^{s+1}\, dz>0$, we have
   \begin{enumerate}
       \item [(a)] if $\int_\Omega g|u|^{1+r}\;dz\leq 0$, then there exists unique $t^-=t^-(u)>t_0$ such that $t^-u\in \mathcal{M}^-_\mu$ and $$I_\mu(t^-u)=\sup_{t\geq 0} I_\mu(tu)>0,$$
       \item [(b)] if $\int_\Omega g|u|^{1+r}\;dz>0$, then there exist unique $t^+=t^+(u),\;t^-=t^-(u)>0$ with $t^+<t_0<t^-$ such that $t^+u\in\mathscr{M}^+_\mu,\;\text{ and } t^-u\in\mathscr{M}^-_\mu$. Moreover,
       $$I_\mu(t^+u)=\inf_{0\leq t\leq t^-}I_\mu(tu),\; \text{ and }I_\mu(t^-u)=\sup_{t\geq t_0}I_\mu(tu).$$
   \end{enumerate}
\end{lem}
\begin{proof}
\begin{enumerate}
   \item[($a$)] Since $\int_\Omega g|u|^{1+r}\;dz\leq 0,$ from Lemma \ref{LM1}, there exists unique $t^->t_0$ such that $$G(t^-)=\textcolor{black}{\mu}\int_\Omega g|u|^{1+r}\;dz \text{ and }G'(t^-)<0.$$
Thus, from (\ref{NM0}) and (\ref{NM1}) we have
$F'(t^-)=0$ and $F''(t^-)<0$. Consequently, $t^-u\in\mathscr{M}^-_\mu$ and $I_\mu(t^-u)=\sup_{t\geq 0}\, I_\mu(tu).$ Now, for every $t>0$, we have 
\begin{align}\label{NM2}
    I_\mu(t^-u)\geq I_\mu(tu)&=\frac{t^2}{2}\int_\Omega |\nabla_\lambda u|^2\, dz-\frac{\mu t^{r+1}}{r+1}\int_\Omega g|u|^{r+1}\, dz-\frac{t^{s+1}}{s+1}\int_\Omega h|u|^{s+1} dz\nonumber\\
    &\geq \frac{t^2}{2}\int_\Omega |\nabla_\lambda u|^2\, dz-\frac{t^{s+1}}{s+1}\int_\Omega h|u|^{s+1} dz.
\end{align}
The function, in the right of the above inequality attained its maximum value at the point $t_1=\left (\frac{\|u\|_\lambda^2}{\int_\Omega h|u|^{s+1}\;dz}\right )^\frac{1}{s-1}$ and the maximum value is $\frac{(s-1)}{2(s+1)}\left(\frac{\|u\|_\lambda^{s+1}}{\int_\Omega h|u|^{1+s}\;dz}\right)^\frac{2}{(s-1)}.$ Hence, from the inequality (\ref{NM2}), we obtain
$$I_\mu(t^-u)\geq \sup_{t\geq 0}\; I_\mu(tu)>0.$$
    \item [($b$)] By using H\"{o}lder inequality, Lemma \ref{Emb} and the identity (\ref{SV}), we obtain
\begin{align}\label{S1}
    0<\mu\int_\Omega g|u|^{1+r}\ dz\leq\mu S_q^{1+r}\|g\|_a\|u\|^{1+r}_\lambda,
\end{align}
and
\begin{align}\label{S2}
 G(t_0)\geq \|u\|_\lambda^{1+r}\left(\frac{s-1}{s-r}\right) \left (\frac{1-r}{s-r}\right )^\frac{1-r}{s-1}\left( S_p^{1+s}\|h\|_b\right )^\frac{-(1-r)}{(s-1)}.
\end{align}
Observing (\ref{S1}) and (\ref{S2}), we have
\begin{align}
 0<\mu\int_\Omega g|u|^{1+r}\ dz&\leq\mu S_q^{1+r}\|g\|_a\|u\|^{1+r}_\lambda\nonumber\\
 &< \|u\|_\lambda^{1+r}\left(\frac{s-1}{s-r}\right) \left (\frac{1-r}{s-r}\right )^\frac{1-r}{s-1}\left( S_p^{1+s}\|h\|_b\right )^\frac{-(1-r)}{(s-1)}\leq G(t_0),
\end{align}
for every $\mu\in(0\; \mu_0)$, where
\begin{align}
    \mu_0=\left(\frac{s-1}{s-r}\right) \left (\frac{1-r}{s-r}\right )^\frac{1-r}{s-1}\frac{\left( S_p^{1+s}\|h\|_b\right )^\frac{-(1-r)}{(s-1)}}{S_q^{1+r}\|g\|_a}.
\end{align}
Thus for every $\mu\in(0,\mu_0)$, we get
\begin{align*}
    G(0)<\mu\int_\Omega g|u|^{1+r}<G(t_0).
\end{align*}
By Lemma \ref{LM1}, there exist unique $t^+,\;t^->0$ with $t^+<t_0<t^-$ such that $G(t^+)=G(t^-)=\mu\int_\Omega g|u|^{1+r} \ dz$ and $G'(t^+)>0,\;G'(t^-)<0$. Hence, (\ref{NM0}) and (\ref{NM1}) suggest that $t^+u\in\mathscr{M}_\mu^+$ and $t^-u\in\mathscr{M}_\mu^-$. Furthermore, since $G$ is increasing in $(0,\;t_0)$ and decreasing in $(t_0,\;\infty)$, so $F(t)$ is decreasing in $(0,\;t^+),\; (t^-,\infty)$ and increasing in $(t^+,\;t^-)$. Thus, $I_\mu(t^+u)=\inf_{0\leq t\leq t^-}I_\mu(tu),\; \text{ and } I_\mu(t^-u)=\sup_{t\geq t_0}I_\mu(tu).$

\end{enumerate}
\end{proof}

\begin{rmk}\label{IMR}
    \begin{enumerate}
        \item[(i)] Due to the condition ($g_1$) and ($h_1$), there exists $u\in H^{1,\lambda}_0(\Omega)$ such that $\int_\Omega h|u|^{s+1}\; dz>0$ and $\int_\Omega g|u|^{1+r}\;dz>0$. Thus, the above Lemma guarantees that $\mathscr{M}^+_\mu$ and $\mathscr{M}^-_\mu$ are nonempty for every $\mu\in(0,\;\mu_0).$
        \item[(ii)] In particular, if $u\in\mathcal{M}^-_\mu\;\mbox{and}\;\int_\Omega h|u|^{s+1}\; dz>0$, then $t^-=1$ and $$I_\mu(t^+u)<I_\mu(u).$$
    \end{enumerate}
\end{rmk}

The following Lemma plays a crucial role in the proof of our argument.

\begin{lem}\label{CL}
      Suppose that $(g_1)$ and $(h_1)$ are satisfied. Then there exists $\Tilde{\mu}>0$ such that for every $\mu\in(0,\Tilde{\mu})$, the following conclusions hold.
      \begin{enumerate}
          \item[(i)]  We have $\mathscr{M}^0_\mu=\emptyset$.
          \item[(ii)] $I_\mu|_{\mathscr{M}_\mu}$ is coercive and bounded below.
          \item[(iii)] $m_\mu\leq m^+_\mu<0.$
      \end{enumerate}
  \end{lem}  
\begin{proof}
    (i) We prove it by contradiction. If possible, let there exist a sequence $\{u_n\}$ such that $u_n\in \mathscr{M}^0_{\mu_n}$, where $\mu_n\to 0$ as $n\to\infty$. Hence, from the definition of $\mathscr{M}^0_{\mu_n}$, we have 
    \begin{align}
        &\int_\Omega |\nabla_\lambda u_n|^2 dz-\mu\int_\Omega g|u_n|^{r+1} dz-\int_\Omega h|u_n|^{s+1} dz=0\label{T31}\\
        and&\nonumber\\
        &2\int_\Omega |\nabla_\lambda u_n|^2 dz-\mu(r+1)\int_\Omega g|u_n|^{r+1} dz-(s+1)\int_\Omega h|u_n|^{s+1} dz=0.\label{T32}
    \end{align}
Eliminating $\int_\Omega g|u_n|^{r+1} dz$, we obtain
\begin{align*}
    (1-r)\int_\Omega |\nabla_\lambda u_n|^2 dz-(s-r)\int_\Omega h|u_n|^{s+1}\ dz=0.
\end{align*}
 Using \textcolor{black}{H\"{o}lder} inequality, Lemma \ref{Emb} and condition ($h_1$), we deduce
\begin{align*}
    (1-r)\|u_n\|_\lambda^2=(s-r)\int_\Omega h|u_n|^{s+1}\ dz&\leq (s-r)\|h\|_b\|u_n\|_{(s+1)b'}^{s+1}\leq (s-r)\|h\|_b\|u_n\|_{p}^{s+1}\\
    &\leq S_p^{s+1}(s-r)\|h\|_{b}\|u_n\|_\lambda^{s+1},
\end{align*}
which implies
 \begin{align}\label{T41}
     \|u_n\|_\lambda\geq \left (\frac{(1-r)}{S_p^{s+1}(s-r)\|h\|_{L^b}}\right )^\frac{1}{s-1}.
 \end{align}
Similarly, by eliminating $\int_\Omega h|u_n|^{s+1} dz$ from (\ref{T31}) and (\ref{T32}), and using \textcolor{black}{Lemma \ref{Emb}} together with ($g_1$), we deduce
\begin{align}\label{T5}
    \|u_n\|_\lambda\leq \left (\frac{S_q^{1+r}\mu_n(s-r)\|g\|_{L^a}}{(s-1)}\right )^\frac{1}{1-r}.
\end{align}
Hence, (\ref{T5}) ensures $\|u_n\|_\lambda\to 0$ as $n\to\infty$, which contradicts (\ref{T41}). Consequently, there exists $\Tilde{\mu}>0$ such that for every $\mu\in(0,\Tilde{\mu})$, $$\mathscr{M}^0_\mu=\emptyset.$$\\
(ii) Let $u\in\mathscr{M}_\mu$. Thus, using (\ref{T31}), we obtain 
\begin{align}
     I_\mu(u)&=\frac{1}{2}\int_\Omega |\nabla_\lambda u|^2 dz-\frac{\mu}{r+1}\int_\Omega g|u|^{r+1} dz-\frac{1}{s+1}\int_\Omega h|u|^{s+1} dz\nonumber\\
     &=\left (\frac{1}{2}-\frac{1}{s+1}\right )\|u\|^2_\lambda-\mu\left (\frac{1}{r+1}-\frac{1}{s+1}\right )\int_\Omega g|u|^{r+1}\ dz\label{RR}\\
     &\geq \left (\frac{1}{2}-\frac{1}{s+1}\right )\|u\|^2_\lambda-\mu\frac{(s-r)}{(r+1)(s+1)}S_q^{1+r}\|g\|_a\|u\|_q^{r+1},\nonumber\\
     &\geq \left (\frac{1}{2}-\frac{1}{s+1}\right )\|u\|^2_\lambda-\mu\frac{(s-r)}{(r+1)(s+1)}S_q^{1+r}\|g\|_a\|u\|_\lambda^{r+1}.\label{T8}
\end{align}
The second last inequality is obtained by using H\"{o}lder inequality, ($g_1$) and \textcolor{black}{Lemma \ref{Emb}}. This proves conclusion (ii).\\
(iii) For $u\in\mathscr{M}_\mu$, \textcolor{black}{we have $\tau(u)=0$ (see \eqref{N2}) and hence}
\begin{align}
    \langle \tau'(u),u\rangle&=2\int_\Omega |\nabla_\lambda u|^2 dz-\mu(r+1)\int_\Omega g|u|^{r+1}\ dz-(s+1)\int_\Omega h|u|^{s+1} dzr\nonumber\\
    &=(1-s)\|u\|_\lambda^2+\mu(s-r)\int_\Omega g|u|^{r+1}\ dz\label{TT9}\\
    &=(1-r)\|u\|_\lambda^2-(s-r)\int_\Omega h|u|^{s+1}\  dz.\label{TT6}
\end{align}
Thus if $u\in\mathscr{M}^+_\mu$ then we have 
\begin{align}\label{T9}
    (s-1)\|u\|_\lambda^2<\mu(s-r)\int_\Omega g|u|^{1+r}\, dz
\end{align}
and
\begin{align}\label{T6}
    (s-r)\int_\Omega h|u|^{1+s}\, dz<(1-r)\|u\|_\lambda^2.
\end{align}
Therefore, utilizing (\ref{T6}), we deduce
\begin{align}\label{SAN}
    I_\mu(u)&=\frac{1}{2}\int_\Omega |\nabla_\lambda u|^2 dz-\frac{\mu}{r+1}\int_\Omega g|u|^{r+1} dz-\frac{1}{s+1}\int_\Omega h|u|^{s+1} dz\nonumber\\
    &=\left (\frac{1}{2}-\frac{1}{r+1}\right )\|u\|^2_\lambda+\left (\frac{1}{r+1}-\frac{1}{s+1}\right )\int_\Omega h|u|^{s+1}\ dz\\
    &<\left (\frac{(r-1)}{2(1+r)}+\frac{(s-r)}{(1+r)(s+1)}\frac{(1-r)}{(s-r)}\right )\|u\|_\lambda^2\nonumber\\
    &=\frac{(1-r)(1-s)}{2(1+r)(s+1)}\|u\|_\lambda^2<0,
\end{align}
which leads to conclude that $m_\mu \leq m^+_\mu<0$.\\
\end{proof}

\begin{rmk}
From (\ref{TT9}) and (\ref{TT6}), one can immediately observe that if $u\in\mathscr{M}_\mu^+$ and $v\in\mathscr{M}_\mu^-$, then $\int_\Omega g|u|^{1+r}\ dz>0$ and $\int_\Omega h|v|^{s+1}\ dz>0$, respectively.
\end{rmk}
The next Lemma helps us to obtain a Palais-Smale sequence in $\mathscr{M}^{+(\text{ or }-)}_\mu$. Motivated by \cite{GT92, WU1, WU}, we prove the following Lemma.
\begin{lem}\label{L}
   For each $u\in\mathscr{M}_\mu^{+(\text{ or }-)}$, there exist $\epsilon>0$ and a differentiable function $F^{+(\text{ or }-)}: B(0,\epsilon)\to \R$ such that $F^{+(\text{ or }-)}(0)=1,\ F^{+(\text{ or }-)}(v)(u-v)\in\mathscr{M}_\mu^{+(\text{ or }-)}$ and $$\langle (F^{+(\text{ or }-)})'(0),v\rangle=\frac{2\int_\Omega \nabla_\lambda u\cdot \nabla_\lambda v\ dz-(1+r)\int_\Omega g|u|^{r-1}uv\ dz-(1+s)\int_\Omega h|u|^{s-1}uv\ dz}{(1-r)\int_\Omega |\nabla_\lambda u|^2\ dz-(s-r)\int_\Omega h|u|^{1+s}\ dz}.$$
\end{lem}
\begin{proof} 
Let us assume that $u\in\mathscr{M}_\mu^+$. We define a function $H:\R\times H^{1,\lambda}_0(\Omega)\to\R$ by
$$H(\eta,v)=\langle I'_\mu(\eta(u-v)),\eta(u-v)\rangle.$$
We also observe $H(1,0)=0$ and 
\begin{align*}
    \frac{d}{d\eta}H(1,0)&=2\int_\Omega |\nabla_\lambda u|^2\, dz-\mu(r+1)\int_\Omega g|u|^{1+r}\ dz-(s+1)\int_\Omega h|u|^{s+1}\ dz\\
    &=(1-r)\int_\Omega |\nabla_\lambda u|^2\, dz-(s-r)\int_\Omega h|u|^{s+1}\ dz>0\textcolor{black}{.}
\end{align*}
Hence, by \textcolor{black}{implicit function theorem}, there exist $\epsilon>0$ and a differentiable function $F^+:B(0,\epsilon)\to\R$ such that $F^+(0)=1$ and $H(F^+(v),v)=0$ for all $v\in B(0,\epsilon)$. Moreover, $$\langle (F^+)'(0),v\rangle=\frac{2\int_\Omega \nabla_\lambda u\cdot \nabla_\lambda v\ dz-(1+r)\int_\Omega g|u|^{r-1}uv\ dz-(1+s)\int_\Omega h|u|^{s-1}uv\ dz}{(1-r)\int_\Omega |\nabla_\lambda u|^2\ dz-(s-r)\int_\Omega h|u|^{1+s}\ dz}.$$
Since $(1-r)\int_\Omega |\nabla_\lambda u|^2\, dz-(s-r)\int_\Omega h|u|^{s+1}\ dz>0$ and the map $u\to \langle\tau'(u),u\rangle$ is continuous, so for sufficiently small $\epsilon>0$ we have for every $v\in B(0,\epsilon)$,
$$(1-r)\int_\Omega |\nabla_\lambda(F^+(v) (u-v)|^2\, dz-(s-r)\int_\Omega h|F^+(v)(u-v)|^{s+1}\ dz>0.$$ Hence, 
$F^+(v)(u-v)\in\mathscr{M}_\mu^+.$ By a similar argument, one can prove the existence of $F^-$.
\end{proof}

\section{ Proof of Theorem \ref{MT1}}\label{PMT1}
 Throughout this section, we always assume $0<\mu<\mu_1:=\min\{ \mu_0,\Tilde{\mu}\}$, where $\mu_0$ and $\Tilde{\mu}$ appeared in Lemma \ref{Char} and Lemma \ref{CL}, respectively. We split our proof into two steps. \\
\textbf{Step-I:} Here we will prove that there exist two sequences $\{u_n\}\subset \mathscr{M}_\mu$ and $\{v_n\}\subset \mathscr{M}^-_\mu$ such that 
$I_\mu(u_n)\to m_\mu,\ \|I'(u_n)\|\to 0$ and $I_\mu(v_n)\to m^-_\mu,\ \|I'(v_n)\|\to 0.$\\
\textbf{proof:} We will prove the existence of $\{v_n\}$ and the other one is similar. \textcolor{black}{To this end, note that for the above choice of $\mu$, $\mathcal{M}_\mu^-$ is closed and hence complete with respect to the inherited topology. Since $I_\mu\big|_{\mathcal{M}_\mu^-}$ is lower semi-continuous and bounded from below, by Ekeland's variational principle \cite[Theorem 1.1]{IE},} there exists a minimizing sequence $\{v_n\}\subset\mathscr{M}^-_\mu$ such that
\begin{align}
    &I_\mu(v_n)\leq m^-_\mu+\frac{1}{n},\\
    \text{ and }&\nonumber\\
    &I_\mu(v_n)\leq I_\mu(w)+\frac{1}{n}\|v_n-w\|_\lambda\quad\mbox{ for all } w\in \mathscr{M}^-_\mu\label{EK}.
\end{align}
Now we show that $\|I'_\mu(v_n)\|\to 0$ as $n\to\infty$. It is enough to prove that there exists a constant $C>0$ (independent of $n$) such that $\langle I'_\mu(v_n),w\rangle\leq\frac{C}{n}$ for every $w\in H^{1,\lambda}_0(\Omega)$ with $\|w\|_\lambda=1$ and for all $n$.\\
Let $w\in H^{1,\lambda}_0(\Omega)$ with $\|w\|_\lambda=1$ be a fixed element and $n\in \N$. Using Lemma \ref{L}, we have a function $F_n: B_{\epsilon_n}(0)\to \R$ such that $F_n(v)(v_n-v)\in \mathscr{M}^-_\mu$ and $$\langle F'_n(0),v\rangle=\frac{2\int_\Omega \nabla_\lambda v_n\cdot \nabla_\lambda v\ dz-\mu(r+1)\int_\Omega g |v_n|^{r-1}v_nv \ dz-(s+1)\int_\Omega h|v_n|^{s-1}v_nv \ dz }{(1-r)\int_\Omega |\nabla_\lambda v_n|^2\ dz-(s-r)\int_\Omega h|v_n|^{s+1}\ dz}.$$
For $0<\delta<\epsilon_n$,\ define $w_\delta=F_n( w\delta )(v_n-w\delta)$. Using (\ref{EK}), we have
\begin{align}
    I_\mu(w_\delta)-I_\mu(v_n)\geq -\frac{\|w_\delta-v_n\|_\lambda}{n},
\end{align}
which implies
\begin{align}\label{T}
    \langle I'_\mu(v_n), w_\delta-v_n\rangle+o(\|w_\delta-v_n\|_\lambda)\geq -\frac{\|w_\delta-v_n\|_\lambda}{n}.
\end{align}
Hence, using the definition of $w_\delta$ and (\ref{EK}), (\ref{T}), we obtain
\begin{align}
    -\delta \langle I'_\mu(v_n), w\rangle&+(F_n(w\delta)-1)\langle I'_\mu(v_n)-I'_\mu(w_\delta),v_n-w_\delta\rangle+(F_n(w\delta)-1)\langle I'_\mu(w_\delta),v_n-w_\delta\rangle \nonumber\\
    &+o(\|w_\delta-v_n\|_\lambda)\geq -\frac{\|w_\delta-v_n\|_\lambda}{n}.
\end{align}
Divide both the sides by $-\delta$, we get 
\begin{align}\label{T1}
    \langle I'_\mu(v_n), w\rangle\leq \frac{(F_n(w\delta)-1)}{\delta}\langle I'_\mu(v_n)&-I'_\mu(w_\delta),v_n-w_\delta\rangle+\frac{(F_n(w\delta)-1)}{\delta}\langle I'_\mu(w_\delta),v_n-w_\delta\rangle \nonumber\\
    &+\frac{1}{\delta}o(\|w_\delta-v_n\|_\lambda)+\frac{\|w_\delta-v_n\|_\lambda}{n\delta}.
\end{align}
 We have $\|w_\delta-v_n\|_\lambda\leq \|(F_n(w\delta)-1)v_n-F_n(w\delta)w\delta\|\leq|(F_n(w\delta)-1)\|v_n\|_\lambda+|F_n(w\delta)|\delta$, which implies
 \begin{align}
     \lim_{\delta\to 0}\frac{\|w_\delta-v_n\|_\lambda}{\delta}\leq \|F'_n(0)\|\|v_n\|_\lambda+1\leq K\|F'_n(0)\|+1.
 \end{align}
Taking $\delta\to 0$ in (\ref{T1}) and using the above fact, we obtain
\begin{align}
  \langle I'_\mu(v_n), w\rangle\leq \frac{C}{n}\left (\|F'_n(0)\|+1\right ).  
\end{align}
Since $w$ is arbitrary, \textcolor{black}{we have
\begin{align}\label{TT2}
  \|I'_\mu(v_n)\|\leq \frac{C}{n}\left (\|F'_n(0)\|+1\right ).  
\end{align} 
Hence, to conclude our claim, it suffices} to show that the sequence $\{\|F'_n(0)\|\}$ is bounded. \textcolor{black}{Form Lemma \ref{L}, Lemma \ref{Emb} and H\"{o}lder inequality,} we have
\begin{align}
    |\langle F'_n(0),v\rangle|\leq \frac{C\|v\|_\lambda}{|(1-r)\|v_n\|_\lambda^2-(s-r)\int_\Omega h|v_n|^{s+1} dz|}.
\end{align}
It is sufficient to prove that the denominator is bounded below by some positive constant. If possible, let up to a subsequence
\begin{align}
    \left |(1-r)\|v_n\|_\lambda^2-(s-r)\int_\Omega h|v_n|^{s+1} dz\right |\to 0 \text{ as } n\to \infty,
\end{align}
that is,
\begin{align}\label{SAN1}
    (1-r)\|v_n\|_\lambda^2-(s-r)\int_\Omega h|v_n|^{s+1} dz=o(1).
\end{align}
For $v_n\in\mathscr{M}^-_\mu$, the identity (\ref{TT6}) leads to conclude that
\begin{align*}
  (1-r)\|v_n\|_\lambda^2<(s-r)\int_\Omega h|v_n|^{s+1} dz.  
\end{align*}
Utilize the condition ($h_1$), H\"{o}lder inequality, Lemma \ref{Emb} and the above inequality to deduce
\begin{align}\label{T7}
    \|v_n\|_\lambda\geq \left (\frac{(1-r)}{S_p^{s+1}(s-r)\|h\|_b}\right )^\frac{1}{s-1}.
\end{align}
Combining (\ref{SAN1}) and (\ref{T7}), we obtain that for large $n\in\N$,
\begin{align}\label{SA3}
    \int_\Omega h|v_n|^{1+s}\, dz\geq \alpha>0, \text{ for some }\alpha>0.
\end{align}
Now, we have
\begin{align}\label{SA2}
I_\mu(v_n)=\frac{1}{2}\int_\Omega |\nabla_\lambda v_n|^2\;dz-\frac{\mu}{1+r}\int_\Omega g|v_n|^{1+r}\;dz-\frac{1}{1+s}\int_\Omega h|v_n|^{1+s}\; dz=m^-_\mu+o(1).    
\end{align}
From (\ref{SAN}), (\ref{SAN1}), and (\ref{SA2}) we get
\begin{align}
  \left \{  \left (\frac{1}{2}-\frac{1}{r+1}\right )\left (\frac{s-r}{1-r}\right )-\left( \frac{1}{s+1}-\frac{1}{r+1}\right )\right \}\int_\Omega h|v_n|^{s+1}=m^-_\mu+o(1)
\end{align}
that is,
\begin{align}\label{SA4}
    \frac{(s-r)(1-s)}{2(1+s)(1+r)}\int_\Omega h|v_n|^{s+1}=m^-_\mu+o(1).
\end{align}
 Combining (\ref{SA3}) and (\ref{SA4}), we obtain $m^-_\mu<0$. For large enough $n,$ we have 
 \begin{align}
     I_\mu(v_n)=\left (\frac{1}{2}-\frac{1}{s+1}\right )\|v_n\|_\lambda^2-\mu\left (\frac{1}{1+r}-\frac{1}{s+1}\right)\int_\Omega g|v_n|^{1+r}\;dz<\frac{m_\mu^-}{2}<0.
 \end{align}
This yields
\begin{align}
    \frac{s-1}{2(s+1)}\|v_n\|^2<\mu\frac{(s-r)}{(s+1)(r+1)}\int_\Omega g|v_n|^{1+r}\;dz\leq \mu\frac{S_q^{1+r}(s-r)}{(s+1)(r+1)}\|g\|_a\|v_n\|_\lambda^{1+r},
\end{align}
that is
\begin{align}\label{SA5}
    \|v_n\|<\left ( \frac{2S_q^{1+r}\mu(s-r)\|g\|_a}{(s-1)(r+1)}  \right )^\frac{1}{1-r}.
\end{align}
Now, since $v_n\in\mathcal{M}^-_\mu$ and (\ref{SAN1}) holds,
\begin{align}\label{SA6}
o(1)&=\left(\frac{1-r}{s-r}\right)^\frac{s}{s-1}\left(\frac{s-1}{1-r}\right)\left(\frac{\|v_n\|_\lambda^{2s}}{\int_\Omega h|v_n|^{1+s}\;dz}\right )^\frac{1}{s-1}-\mu\int_\Omega g|v_n|^{1+r}\;dz\nonumber\\
&\geq \left(\frac{1-r}{s-r}\right)^\frac{s}{s-1}\left(\frac{s-1}{1-r}\right)\left(\frac{\|v_n\|_\lambda^{2s}}{S_p^{s+1}\|h\|_b\|v_n\|_\lambda^{1+s}}\right )^\frac{1}{s-1}-S_q^{1+r}\mu\|g\|_a\|v_n\|_\lambda^{1+r}\nonumber\\
&\geq \|v_n\|_\lambda^{1+r}\left (      \left(\frac{1-r}{s-r}\right)^\frac{s}{s-1}\left(\frac{s-1}{1-r}\right)\left(\frac{1}{S_p^{1+r}\|h\|_b}  \right )^\frac{1}{s-1}\|v_n\|_\lambda^{-r}-\mu S_q^{1+r}\|g\|_a  \right )\nonumber\\
&\geq \|v_n\|_\lambda^{1+r}\left (      \left(\frac{1-r}{s-r}\right)^\frac{s}{s-1}\left(\frac{s-1}{1-r}\right)\left(\frac{1}{S_p^{s+1}\|h\|_b}  \right )^\frac{1}{s-1}\left ( \frac{2\mu S_q^{1+r}(s-r)\|g\|_a}{(s-1)(r+1)}  \right )^\frac{-r}{1-r}-\mu S_q^{1+r}\|g\|_a  \right ).
\end{align}
In the last inequality, we used the inequality (\ref{SA5}). There exists $\mu_2>0$ such that for every $\mu<\mu_2,$ $$\left (      \left(\frac{1-r}{s-r}\right)^\frac{s}{s-1}\left(\frac{s-1}{1-r}\right)\left(\frac{1}{S_p^{1+s}\|h\|_b}  \right )^\frac{1}{s-1}\left ( \frac{2\mu S_q^{1+r}(s-r)\|g\|_a}{(s-1)(r+1)}  \right )^\frac{-r}{1-r}-\mu S_q^{1+r}\|g\|_a  \right )>0.$$
Thus, for every $\mu<\mu_3:=\min\{\mu_1,\mu_2\}$, using (\ref{T7}), (\ref{SA6}) we deduce

$$o(1)=\left(\frac{1-r}{s-r}\right)^\frac{s}{s-1}\left(\frac{s-1}{1-r}\right)\left(\frac{\|v_n\|_\lambda^{2s}}{\int_\Omega h|v_n|^{1+s}\;dz}\right )^\frac{1}{s-1}-\mu\int_\Omega g|v_n|^{1+r}\;dz\geq \beta>0,$$
\textcolor{black}{for some $\beta>0$, which is a contradiction}. Hence, there exists a sequence $\{v_n\}\subset\mathcal{M}^-_\mu$ such that $$I_\mu(v_n)\to m^-_\mu \text{ and } \|I_\mu(v_n)\|\to 0 \text{ as }n\to\infty.$$
In order to prove the existence of the sequence $\{u_n\}$, 
by Ekeland's variational principle, there exists a minimizing sequence $\{u_n\}\subset\mathscr{M}_\mu$ such that
\begin{align}
    &I_\mu(u_n)\leq m_\mu+\frac{1}{n},\\
    \text{ and }&\nonumber\\
    &I_\mu(u_n)\leq I_\mu(w)+\frac{1}{n}\|u_n-w\|_\lambda\mbox{ for all } w\in \mathscr{M}_\mu\label{EK1}.
\end{align}
Due to the fact (\ref{RR}), for large enough $n,$

$$\left (\frac{1}{2}-\frac{1}{s+1}\right )\|u_n\|^2_\lambda-\mu\left (\frac{1}{r+1}-\frac{1}{s+1}\right )\int_\Omega g|u_n|^{r+1}\ dz<\frac{m_\mu}{2}<0.$$
Thus, we get 
\begin{align}
    &\int_\Omega g|u_n|^{1+r}\;dz>\frac{-m_\mu(1+r)(s+1)}{2(s-r)}>0,\label{RT1}\\
    \text{ and }&\nonumber\\
    &\left (\frac{1}{2}-\frac{1}{s+1}\right )\|u_n\|^2_\lambda\leq \mu\left (\frac{1}{r+1}-\frac{1}{s+1}\right )\int_\Omega g|u_n|^{r+1}\ dz.\label{RT2}
\end{align}
These facts yield 
$$\|u_n\|_\lambda\geq \left(\frac{-m_\mu(1+r)(s+1)}{2(s-r)S_q^{1+r}\|g\|_a}\right )^\frac{1}{1+r} \text{ and }\|u_n\|\leq \left (\frac{2\mu(s-r)S_q^{1+r}\|g\|_a}{(r+1)(s-1)}\right )^\frac{1}{1-r}.$$
The rest of the proof is exactly the same as in the previous one.

\textbf{Step-II:}
We claim that the minimizers $m_\mu$ and $m^-_\mu$ are achieved by two critical points of the function $I_\mu$.\\
\textbf{Proof:} \textcolor{black}{First, we will show that $m^-_\mu$ is achieved by a critical point.}\\
\textbf{ First Solution:}
Due to Lemma \ref{CL}(ii), the above minimizing sequence $\{v_n\}$ is bounded. Hence, there exists $v^*\in H^{1,\lambda}_0(\Omega)$ such that up to a sub-sequence the following holds:
\begin{align}
    \begin{cases}
        v_n\rightharpoonup v^* \text{ weakly in } H^{1,\lambda}_0(\Omega),\\
        v_n\to v^* \text{ strongly in } L^p(\Omega) \text{ for } 1\leq p<2^*_\lambda,\\
        v_n\to v^* \text{ pointwise a.e in } \Omega.
    \end{cases}
\end{align}
Using these facts, we have 
\begin{align}\label{SOL}
    \langle I'_\mu(v^*),\phi\rangle=\lim_{n\to\infty} \langle I'_\mu(v_n),\phi\rangle=0 \text{ for all } \phi\in C^\infty_c(\Omega).
\end{align}
Hence, $v^*$ is a critical point of $I_\mu$. Now we claim that $v_n\to v^*$ in $H^{1,\lambda}_0(\Omega)$. It is enough to prove $\|v_n\|_\lambda\to\|v^*\|_\lambda$ as $n\to\infty.$ If $\|v^*\|_\lambda<\liminf_{n\to\infty}\|v_n\|_\lambda$, then we obtain
\begin{align*}
   \langle I'_\mu(v^*),v^*\rangle=&\int_\Omega |\nabla_\lambda v^*|^2\ dz-\mu\int_\omega g|v^*|^{r+1}\ dz-\int_\Omega h|v^*|^{s+1}\ dz\\
   &<\liminf_{n\to\infty}\left [ \int_\Omega |\nabla_\lambda v_n|^2\ dz-\mu\int_\omega g|v_n|^{r+1}\ dz-\int_\Omega h|v_n|^{s+1}\ dz \right ]\\
   &=\liminf_{n\to\infty} \langle I'_\mu(v_n),v_n\rangle=0,
\end{align*}
which contradicts the equality (\ref{SOL}). Consequently, $v_n\to v^*$ in $H^{1,\lambda}_0(\Omega)$. Moreover, $$\langle\tau'(v^*),v^*\rangle=\lim \langle\tau'(v_n),v_n\rangle\leq 0.$$ Thus, $v^*\in\mathcal{M}^-_\mu$ is a critical point of $I_\mu$ with $I_\mu(v^*)=m^-_\mu$. \textcolor{black}{Since $\mathcal{M}^-_\mu\subset \mathcal{M}_\mu$, from the definition of $\mathcal{M}_\mu$ (see \ref{N1}), one has $v^*\not\equiv 0$.} Hence, if necessary, by taking $|v^*|$, we can conclude that $v^*$ is a non-trivial, non-negative solution of \textcolor{black}{problem (\ref{ME})}.\\
\textbf{Second  Solution:}
Let $\{u_n\}$ be the minimizing sequence obtained in Step-I. By Lemma \ref{CL}, $\{u_n\}$ is bounded in $H^{1,\lambda}_0(\Omega)$, and hence there exists $u^*\in H^{1,\lambda}_0(\Omega)$ such that the following holds (up to a sub-sequence):  
\begin{align}
    \begin{cases}
        u_n\rightharpoonup u^* \text{ weakly in } H^{1,\lambda}_0(\Omega),\\
        u_n\to u^* \text{ strongly in } L^p(\Omega) \text{ for } 1\leq p<2^*_\lambda,\\
        u_n\to u^* \text{ pointwise a.e in } \Omega.
    \end{cases}
\end{align}
It follows from these facts that for all  $\phi\in C^\infty_c(\Omega)$, we have
\begin{align}\label{SOLL}
    \langle I'_\mu(u^*),\phi\rangle=\lim_{n\to\infty} \langle I'_\mu(u_n),\phi\rangle=0.
\end{align}
Consequently, $u^*$ is a critical point of $I_\mu$, and hence $u^*\in M_\mu$. Now,
applying $\|u^*\|_\lambda\leq\liminf\|u_n\|_\lambda$, we deduce
\begin{align}
   m_\mu=\liminf I_\mu(u_n)&=\liminf\left ( \frac{1}{2}\int_\Omega|\nabla_\lambda u_n|^2\, dz-\frac{\mu}{1+r}\int_\Omega g|u_n|^{1+r}\, dz-\frac{1}{1+s}\int_\Omega h|u_n|^{1+s}\, dz\right )\nonumber\\
    &\geq \frac{1}{2}\int_\Omega|\nabla_\lambda u^*|^2\, dz-\frac{\mu}{1+r}\int_\Omega g|u^*|^{1+r}\, dz-\frac{1}{1+s}\int_\Omega h|u^*|^{1+s}\, dz\nonumber\\
    &=I_\mu(u^*)\geq m_\mu.
\end{align}
Thus, $u^*\in\mathscr{M}_\mu$ is a solution to \textcolor{black}{problem (\ref{ME})} such that $I_\mu(u^*)=m_\mu$. Furthermore, $u^*\in\mathcal{M}^+_\mu,$ otherwise if $u^*\in\mathcal{M}^-_\mu,$ then by Lemma \ref{Char} together with (\ref{RT1}), there exists $t^+<1$ such that $t^+u^*\in\mathcal{M}^+_\mu$ and $$I_\mu(t^+u^*)<I_\mu(u^*)=m_\mu,$$
which is a contradiction. Hence $u^*\in\mathscr{M}^+_\mu$ and $I_\mu(u^*)=m_\mu=m^+_\mu.$ \textcolor{black}{Since $\mathcal{M}^+_\mu\subset \mathcal{M}_\mu$, from the definition of $\mathcal{M}_\mu$ (see \ref{N1}), we can conclude that $u^*\not\equiv 0$.}
Hence, $u^*,\; v^*\in H^{1,\lambda}_0(\Omega)$ are two non-trivial, nonnegative solutions of (\ref{ME}).

In particular, if $n=1,\lambda\geq 1;$ and $g,h$ are non-negative then \cite[Theorem 2.6]{KS} ensures that the solutions $u^*$ and $v^*$ are strictly positive in $\Omega.$ This finishes the proof of Theorem \ref{MT1}.
\section{ Proof of Theorem \ref{MT2}}\label{PMT2}
In the sequel, we always assume that all the assumptions of Theorem \ref{MT2} are satisfied and $I_\mu$ is the functional defined in (\ref{CF}). That is, $I_\mu: H^{1,\lambda}_0(\Omega)\to\R$, defined as 
 \begin{align}\label{CF1}
     I_\mu(u)=\frac{1}{2}\int_\Omega |\nabla_\lambda u|^2\ dz-\frac{\mu}{r+1}\int_\Omega g\cdot (u^+)^{r+1}\ dz - \frac{1}{2^*_\lambda}\int_\Omega (u^+)^{2^*_\lambda}\ dz.
 \end{align}
 Observe that $\langle I'_\mu(u),u^-\rangle=-\|u^-\|_\lambda^2$, and this leads to conclude that every critical point of $I_\mu$ is non-negative. The following lemma provides the first solution to the \textcolor{black}{ problem (\ref{ME1})}.
\begin{lem}\label{Fsol}
     There exists $\mu^*>0$ such that for every $\mu\in(0,\mu^*)$, the \textcolor{black}{problem (\ref{ME1})} admits at least one non-zero solution $u_\mu$ with $I_\mu(u_\mu)<0$.
\end{lem}
\begin{proof}
Using H\"{o}lder inequality and Lemma \ref{Emb}, we have 
\begin{align}
    I_\mu(u)&=\frac{1}{2}\int_\Omega |\nabla_\lambda u|^2\ dz-\frac{\mu}{r+1}\int_\Omega g\cdot(u^+)^{r+1}\ dz - \frac{1}{2^*_\lambda}\int_\Omega h\cdot(u^+)^{2^*_\lambda}\, dz\nonumber\\
     &\geq \frac{1}{2}\|u\|_\lambda^2-\frac{\mu S_q^{1+r}}{r+1}\|g\|_a\|u\|_\lambda^{1+r}-\frac{S_p^{2^*_\lambda}\|h\|_b}{2^*_\lambda}\|u\|_\lambda^{2^*_\lambda}.
\end{align}
Let us fix a small $\delta>0$ such that $\frac{\delta^2}{4}-\frac{S_p^{2^*_\lambda}\|h\|_b\delta^{2^*_\lambda}}{2^*_\lambda}>0$, and choose $\mu^*=\frac{(1+r)\delta^{1-r}}{8\|g\|_aS_q^{1+r}}.$ If $\mu\in(0, \mu^*)$, then we have
\begin{align}\label{MP1}
    I_\mu(u)\geq
    \begin{cases}
    \frac{\delta^2}{8}>0 \text{ if } \|u\|_\lambda=\delta\\
    -\alpha(\delta) \text{ if } \|u\|_\lambda\leq\delta,
    \end{cases}
\end{align}
for some $\alpha(\delta)>0$. For any non-negative $u\in H^{1,\lambda}_0(\Omega)\setminus{0}$ and $t>0$, we have
\begin{align}
    I_\mu(tu)&=\frac{t^2}{2}\int_\Omega |\nabla_\lambda u|^2\ dz-\frac{\mu t^{r+1}}{r+1}\int_\Omega g |u|^{r+1}\ dz - \frac{t^{2^*_\lambda}}{2^*_\lambda}\int_\Omega |u|^{2^*_\lambda}\ dz\nonumber\\
    &<0,\ \text{ for small enough t. }
\end{align}
Hence, we denote $B_\delta:=\{u\in H^{1,\lambda}_0(\Omega):\ \|u\|_\lambda\leq\delta\}$ and by the above discussion it is clear that
\begin{align}
    -\infty<\beta:=\inf_{B_\delta} I_\mu(u)<0.
\end{align}
By Ekeland's variational principle, there exists a sequence $\{u_n\}\subset B_\delta$ such that $$I_\mu(u_n)\to \beta\text{ and } \|I'_\mu(u_n)\|\to 0\text{ as } n\to\infty.$$
Since $\{u_n\}$ is bounded, there exists $u_\mu\in B_\delta$ such that up to a sub-sequence, the following holds:
\begin{align}\label{T3}
    \begin{cases}
        u_n\rightharpoonup u_\mu \text{ weakly in } H^{1,\lambda}_0(\Omega)\\
        u_n\to u_\mu \text{ strongly in } L^p(\Omega) \text{ for } 1\leq p<2^*_\lambda\\
        u_n\to u_\mu \text{ pointwise a.e. in } \Omega.
    \end{cases}
\end{align}
Since $\{u_n^+\}$ is also bounded in $H^{1,\lambda}_0(\Omega)$, so up to a subsequence $u_n^+\to u_\mu^+$ strongly in $L^p(\Omega)$ for all $1\leq p<2^*_\lambda$, and  $(u_n^+)^{2^*_\lambda-1}\rightharpoonup (u_\mu^+)^{2^*_\lambda-1}$ weakly in $L^\frac{2^*_\lambda}{2^*_\lambda-1}(\Omega)$. It follows from these facts that 
\begin{align}
    o(1)=\langle I'_\mu(u_n),\phi\rangle=\langle I'_\mu(u_\mu),\phi\rangle+o(1),\text{ for all } \phi\in C^\infty_c(\Omega).
\end{align}
Therefore, $u_\mu$ is a critical point of $I_\mu.$ Applying Fatou's Lemma, we also obtain
\begin{align}
   \beta&=\liminf_{n\to\infty} \left ( I_\mu(u_n)-\frac{1}{2^*_\lambda}\langle I'_\mu(u_n),u_n\rangle\right )\nonumber\\
   &=\liminf_{n\to\infty}\left ( (\frac{1}{2}-\frac{1}{2^*_\lambda})\int_\Omega |\nabla_\lambda u_n|^2-(\frac{1}{1+r}-\frac{1}{2^*_\lambda})\mu\int_\Omega g\cdot(u^+_n)^{1+r}\right )\nonumber\\
    &\geq(\frac{1}{2}-\frac{1}{2^*_\lambda})\int_\Omega |\nabla_\lambda u_\mu|^2-(\frac{1}{1+r}-\frac{1}{2^*_\lambda})\mu\int_\Omega g\cdot(u^+_\mu)^{1+r}\nonumber\\
    &=I_\mu(u_\mu)-\frac{1}{2^*_\lambda}\langle I'_\mu(u_\mu),u_\mu\rangle= I_\mu(u_\mu)\geq \beta.
\end{align}
Consequently, $u_\mu\in\mathring{B}_\delta$ is a solution of (\ref{ME1}) with $I_\mu(u_\mu)=\beta<0.$ \textcolor{black}{Clearly, $u_\mu\not\equiv 0$, otherwise $\beta=0$.}
\end{proof}
We define 
\begin{align}\label{d1}
    \Tilde{c}=\inf\{\, I_\mu(v)+\frac{1}{Q}\mathcal{S}_\lambda^\frac{Q}{2}: v\in H^{1,\lambda}_0(\Omega) \text{ such that } I'_\mu(v)=0 \}
\end{align}
We will prove that the functional $I_\mu$ satisfies the Palais-Smale condition (abbreviated as $(PS)_c$) below the level $\Tilde{c}$, and here we use the argument of Brezis and Nirenberg \cite{BN}. 
\begin{lem}\label{BrezisNiren1}
     The functional $I_\mu$ satisfies $(PS)_c$ condition for every $c<\Tilde{c}$ and $\mu\in(0,\mu^*)$. Here $\mu^*$ is obtained in Lemma \ref{Fsol}.
\end{lem}

\begin{proof}
Let $c<\Tilde{c}$ and $\{u_n\}$ be a sequence in $H^{1,\lambda}_0(\Omega)$ such that 
\begin{align}\label{W5}
    \begin{cases}
        I_\mu(u_n)\to c \text{ as } n\to\infty\\
        I'_\mu(u_n)\to 0 \text{ in } (H^{1,\lambda}_0(\Omega))^*.
    \end{cases}
\end{align}
Applying this fact, we deduce that for large enough $n\in\N,$
\begin{align*}
    \frac{1}{Q}\|u_n\|^2_\lambda-\mu\left (\frac{1}{r+1}-\frac{1}{2^*_\lambda}\right )\int_\Omega g\cdot(u_n^+)^{r+1}=I_\mu(u_n)-\frac{1}{1+s}\langle I'_\mu(u_n),u_n\rangle\leq c+1+m\|u_n\|_\lambda,
    \end{align*}
  for some $m>0$. Using H\"{o}lder inequality, we obtain
    \begin{align}\label{BDD}
         \frac{1}{Q}\|u_n\|_\lambda^2-\mu\left (\frac{1}{r+1}-\frac{1}{2^*_\lambda}\right )\|g\|_a\|u_n\|_\lambda^{1+r}\leq c+1+m\|u_n\|_\lambda,
    \end{align}
   which yields that the sequence $\{u_n\}$ is bounded in $H^{1,\lambda}_0(\Omega)$. Consequently, we can assume that there exists $u\in H^{1,\lambda}_0(\Omega)$ such that up to a subsequence, the following holds.
   \begin{align}\label{Lab}
        \begin{cases}
       u_n\rightharpoonup u \text{ weakly in }H^{1,\lambda}_0(\Omega),\\
        u_n\to u \text{ strongly in } L^{q}(\Omega) \text{ for all }1\leq q< 2^*_\lambda,\\
        u_n\to u \text{ pointwise a.e. in }\Omega.
    \end{cases}
   \end{align}
 Using this fact, one has
\begin{align}\label{W3}
    o(1)=\langle I'_\mu(u_n),\phi\rangle=\langle I'_\mu(u),\phi\rangle+o(1) \text{ for all }\phi\in H^{1,\lambda}_0(\Omega).
\end{align}
Thus, $u$ is a critical point of $I_\mu$, and hence non-negative. Since $c<\Tilde{c}$, from the definition of $\Tilde{c}$, we have $c<I_\mu(u)+\frac{1}{Q}\mathcal{S}_\lambda^\frac{Q}{2}$. Next we prove that the sequence $\{u_n\}$ has a convergent subsequence in $H^{1,\lambda}_0(\Omega)$. We have
\begin{align}
    &\|u_n\|_\lambda^2=\|u_n-u\|_\lambda^2+\|u\|_\lambda^2+o(1)\label{W1},\\
    \text{ and }&\nonumber\\
    &\langle I_\mu'(u_n),u_n\rangle=\int_\Omega |\nabla_\lambda u_n|^2\, dz-\mu\int_\Omega g\cdot(u_n^+)^{1+r} \, dz-\int_\Omega(u_n^+)^{2^*_\lambda} dz=o(1).\label{TA1}
\end{align}
Thus, it is immediate from the facts (\ref{TA1}) and $\langle I'_\mu(u),u\rangle=0$ that
\begin{align}\label{W2}
    \int_\Omega |\nabla_\lambda u_n|^2\, dz-\int_\Omega|u_n|^{2^*_\lambda} dz=\mu\int_\Omega g(u^+)^{1+r} \, dz + o(1)=\int_\Omega |\nabla_\lambda u|^2\, dz-\int_\Omega(u^+)^{2^*_\lambda} dz+o(1).
\end{align}
Since $\{u_n^+\}$ is bounded in $L^{2^*_\lambda}(\Omega)$, applying Brezis-Lieb Lemma, one has 
\begin{align}\label{W4}
  \|u_n^+\|_{2^*_\lambda}^{2^*_\lambda} = \|(u_n-u)^+\|_{2^*_\lambda}^{2^*_\lambda}+\|u^+\|_{2^*_\lambda}^{2^*_\lambda}+o(1).
\end{align}
The identity (\ref{W1}), (\ref{W2}) and (\ref{W4}) imply that
\begin{align}\label{W8}
    \|u_n-u\|_\lambda^2=\|(u_n-u)^+\|_{2^*_\lambda}^{2^*_\lambda}+o(1)
\end{align}
We substitute (\ref{W1}) and (\ref{W4}) in the expression of $I_\mu(u_n)$ and using the fact (\ref{W5}), we obtain 
\begin{align*}
    \frac{1}{2}\left [\|u\|_\lambda^2+\|u_n-u\|_\lambda^2 \right ]-\frac{\mu}{r+1}\int_\Omega g\cdot(u_n^+)^{1+r}\, dz-\frac{1}{1+s}\left [\|(u_n-u)^+\|_{2^*_\lambda}^{2^*_\lambda}+\|u^+\|_{2^*_\lambda}^{2^*_\lambda}\right ]=c+o(1),
\end{align*}
which implies
\begin{align}\label{W9}
    I_\mu(u)+\left [\frac{1}{2}\|u_n-u\|_\lambda^2-\frac{1}{1+s} \|(u_n-u)^+\|_{2^*_\lambda}^{2^*_\lambda} \right ]=c+o(1).
\end{align}
Suppose $u_n$ does not converge to $u$ in $H^{1,\lambda}_0(\Omega)$, i.e., $\|u_n-u\|_\lambda\to d\neq 0$. From the definition of $\mathcal{S}_\lambda$, we get

\begin{align}
    \|u_n-u\|_\lambda^2\geq \mathcal{S}_\lambda\|u_n-u\|^2_{2^*_\lambda}\geq \mathcal{S}_\lambda\|(u_n-u)^+\|^2_{2^*_\lambda}
\end{align}
This fact, together with the identity (\ref{W8}), ensures that $d$ enjoys the following inequality:
$$d^2\geq \mathcal{S}_\lambda d^\frac{4}{2^*_\lambda}.$$
Thus, $d\geq \mathcal{S}_\lambda^\frac{Q}{4}$. Using this fact, we can infer from the identity (\ref{W9}) that $$c=I_\mu(u)+\frac{1}{Q}d^2\geq I_\mu(u)+\frac{1}{Q}\mathcal{S}_\lambda^\frac{Q}{2}$$
which is a contradiction as $c<I_\mu(u)+\frac{1}{Q}\mathcal{S}_\lambda^\frac{Q}{2}$. Consequently, $u_n\to u$ strongly in $H^{1,\lambda}_0(\Omega)$.
\end{proof}
We mention an important remark that will be useful in proving our main theorem.
\begin{rmk}
    If $u\in H^{1,\lambda}_0(\Omega)$ is the unique non-zero critical point with negative energy then $I_\mu$ satisfies $(PS)_c$ condition for every $c<I_\mu(u)+\frac{1}{Q}\mathcal{Q}^\frac{Q}{2}$ and $\mu\in(0,\mu^*)$.
\end{rmk}
In order to show that the functional exhibits Mountain Pass Geometry, the following result is essential. This theorem requires the assumption ($g_2$).
\begin{lem}\label{MPL}
    We assume that all the assumptions of Theorem \ref{MT2} are satisfied and $u_\mu$ is the solution of \textcolor{black}{problem (\ref{ME1})} obtained in Lemma \ref{Fsol}. Then for every $\mu\in(0,\mu^*)$, there exists $\epsilon>0$ such that 
    \begin{align}
        \max_{t>0}I_\mu(u_\mu+tw_\epsilon) < I_\mu(u_\mu)+\frac{1}{Q}\mathcal{S}_\lambda^\frac{Q}{2},
    \end{align}
    where $w_\epsilon$ defined in (\ref{Asy}) and $\mu^*$ is obtained in Lemma \ref{Fsol}.
\end{lem}
\begin{proof}
Since the function $g$ satisfies $(g_2)$ and the Grushin operator $\Delta_\lambda$ is uniformly elliptic away from $x=0$ plane, well-established elliptic theory guarantees that $u_\mu\in L^\infty(B_{2R}(z_0))$ (this ball is defined in condition $g_2$). We choose this specific ball and define the functions $u_\epsilon$ and $w_\epsilon$, as introduced in the preliminary section (see (\ref{Asy})). Since $u_\mu$ and $w_\epsilon$ are non-negative, we have
\begin{align}
    I_\mu(u_\mu+tw_\epsilon)&=\frac{1}{2}\int_\Omega |\nabla_\lambda u_\mu+t \nabla_\lambda w_\epsilon|^2-\frac{\mu}{1+r}\int_\Omega g\cdot (u_\mu+tw_\epsilon)^{1+r}\, dz\nonumber\\
    &-\frac{1}{2^*_\lambda}\int_\Omega  (u_\mu+tw_\epsilon)^{2^*_\lambda}\, dz\nonumber\\
    &=\frac{1}{2}\int_\Omega|\nabla_\lambda u_\mu|^2+\frac{t^2}{2}\int_\Omega|\nabla_\lambda w_\epsilon|^2\, dz+t\int_\Omega \nabla_\lambda u_\mu\cdot \nabla_\lambda w_\epsilon dz\nonumber\\
    &-\frac{\mu}{1+r}\int_\Omega g\cdot (u_\mu+tw_\epsilon)^{1+r}\, dz
    -\frac{1}{2^*_\lambda}\int_\Omega  (u_\mu+tw_\epsilon)^{2^*_\lambda}\, dz
\end{align}
Using \cite[ Lemma 2.4]{Silva} with $\sigma=\frac{2^*_\lambda}{2}$, i.e., $(1+t)^{2^*_\lambda}\geq 1+t^{2^*_\lambda}+{2^*_\lambda}t+2^*_\lambda t^{2^*_\lambda-1}-Ct^\frac{Q}{Q-2}$; the fact $\langle I'_\mu(u_\mu),tw_\epsilon\rangle=0$, and $\|w_\epsilon\|_{2^*_\lambda}=1$, we deduce
\begin{align}
   I_\mu(u_\mu+tw_\epsilon)&\leq I_\mu(u_\mu)+\frac{t^2}{2}\int_\Omega|\nabla_\lambda w_\epsilon|^2\, dz-\frac{t^{2^*_\lambda}}{2^*_\lambda}\nonumber\\
    &-\frac{\mu}{1+r}\int_\Omega g\cdot\left ((u_\mu+tw_\epsilon)^{1+r}-u_\mu^{1+r}-(1+r)u_\mu^rtw_\epsilon\right )\, dz\nonumber\\
    &-t^{2^*_\lambda-1}\int_\Omega u_\mu w_\epsilon^{2^*_\lambda-1}\, dz+\frac{Ct^\frac{Q}{Q-2}}{2^*_\lambda}\int_\Omega (u_\mu w_\epsilon)^\frac{Q}{Q-2}\, dz.
\end{align}
Since $\left ((u_\mu+tw_\epsilon)^{1+r}-u_\mu^{1+r}-(1+r)u_\mu^rtw_\epsilon\right )\geq 0$ and $g\geq 0$, we have
\begin{align}\label{san}
    I_\mu(u_\mu+tw_\epsilon)&\leq  I_\mu(u_\mu)+\frac{t^2}{2}\int_\Omega|\nabla_\lambda w_\epsilon|^2\, dz-\frac{t^{2^*_\lambda}}{2^*_\lambda}-t^{2^*_\lambda-1}\int_{\{z\in\Omega: u_\mu(z)\leq 1\}} u_\mu w_\epsilon^{2^*_\lambda-1}\, dz\nonumber\\
    &+\frac{Ct^\frac{Q}{Q-2}}{2^*_\lambda}\int_\Omega (u_\mu w_\epsilon)^\frac{Q}{Q-2}\, dz.
\end{align}
 Since $u_\mu$ is bounded in $B_{2R}(z_0)$, applying proposition \ref{Pro1} with $\gamma=2^*_\lambda-1$ and $\gamma=\left (\frac{Q}{Q-2}\right )$ respectively, we obtain
\begin{align}\label{san2}
    \int_{\{z\in\Omega: u_\mu(z)\leq 1\}} u_\mu w_\epsilon^{2^*_\lambda-1}\, dz\leq\int_\Omega w_\epsilon^{2^*_\lambda-1}\, dz=O(\epsilon^\frac{Q-2}{2})
\end{align}
and
\begin{align}\label{san3}
    \int_\Omega (u_\mu w_\epsilon)^\frac{Q}{Q-2}\, dz\leq O(\epsilon^\frac{Q}{2}).
\end{align}
We also deduce from the proposition \ref{Pro1} that
\begin{align}
    \int_\Omega|\nabla_\lambda w_\epsilon|^2\leq \mathcal{S}_\lambda+ O(\epsilon^{Q-2}).
\end{align}
Combining these facts, (\ref{san}) yields
\begin{align}\label{A}
    I_\mu(u_\mu+tw_\epsilon)&\leq  I_\mu(u_\mu)+\underbrace{\frac{\mathcal{S}_\lambda t^2}{2}-\frac{t^{2^*_\lambda}}{2^*_\lambda}+C_1t^2\epsilon^{Q-2}-C_2t^{2^*_\lambda-1} \epsilon^\frac{Q-2}{2}+C_3t^\frac{Q}{Q-2} \epsilon^\frac{Q}{2}}_{F_\epsilon(t)},
\end{align}
for some constants $C_1,\;C_2,\;C_3>0.$ Observe that $F_\epsilon(t)\to 0$ and $-\infty$ as $t\to 0$ and $\infty$ respectively, and thus there exists $t_\epsilon>0$ such that $F_\epsilon(t_\epsilon)=\max_{t\geq 0}F_\epsilon(t)$. It is also clear from the fact (\ref{MP1}) that $\{t_\epsilon\}$ is bounded below by some $t_*>0$. Our next claim is that the above sequence is bounded above also. If possible, let up to a subsequence $t_{\epsilon}\to\infty$ as $\epsilon\to 0$. Since $F'_{\epsilon}(t_\epsilon)=0$, we conclude that
\begin{align}
\mathcal{S}_\lambda t_\epsilon+2C_1t_\epsilon\epsilon^{Q-2}+\frac{Q}{Q-2}C_3t_\epsilon^\frac{2}{Q-2}\epsilon^\frac{Q}{2}=t_\epsilon^{2^*_\lambda-1}+{2^*_\lambda}C_2t_\epsilon^{2^*_\lambda-2} \epsilon^\frac{Q-2}{2}.    
\end{align}
After dividing both sides by $t_\epsilon^{2^*_\lambda-1}$ and taking $\epsilon\to 0$, we obtain a contradiction. Thus, \{$t_\epsilon$\} is bounded above by some $t^*>0$. Thus, the inequality (\ref{A}) ensures that for small $\epsilon>0$, we have
\begin{align}
   \max_{t\geq 0} I_\mu(u_\mu+tw_\epsilon)&\leq I_\mu(u_\mu)+\max_{t\geq 0}\left( \frac{\mathcal{S}_\lambda t^2}{2}-\frac{t^{2^*_\lambda}}{2^*_\lambda}\right )\nonumber\\
    &+\underbrace{C_1(t^*)^2\epsilon^{Q-2}-C_2(t_*)^{2^*_\lambda-1} \epsilon^\frac{Q-2}{2}+C_3(t^*)^\frac{Q}{Q-2} \epsilon^\frac{Q}{2}}_{=H(\epsilon)}\nonumber\\
    &< I_\mu(u_\mu)+\frac{1}{Q}\mathcal{S}_\lambda^\frac{Q}{2}.
\end{align}
 We use the facts $\max_{t\geq 0}\left( \frac{\mathcal{S}_\lambda t^2}{2}-\frac{t^{2^*_\lambda}}{2^*_\lambda}\right )=\frac{1}{Q}\mathcal{S}_\lambda^\frac{Q}{2}$ and $H(\epsilon)<0$ for small enough $\epsilon>0$. 
\end{proof}
\textbf{Proof of Theorem \ref{MT2}:}
We prove this theorem by contradiction. Suppose $\mu\in(0,\mu^*)$ and $u_\mu$ is the only non-zero critical point of the functional $I_\mu$, which is obtained in Lemma \ref{Fsol}. Since $I_\mu(u_\mu)<0$, from the definition of $\Tilde{c}$, we have
\begin{align}
    \Tilde{c}=I_\mu(u_\mu)+\frac{1}{Q}\mathcal{S}_\lambda^\frac{Q}{2}.
\end{align}
We fix the function $w_\epsilon$ obtained in Lemma \ref{MPL}. From the fact (\ref{san}), we observe that $\lim_{t\to\infty}I_\mu(u_\mu+tw_\epsilon)=-\infty$, and hence there exists $T>0$ such that $I_\mu(u_\mu+Tw_\epsilon)<I_\mu(u_\mu)$ and $\|u_\mu+Tw_\epsilon\|>\delta$. Define
 $$\Gamma:=\{\, \gamma\;| \gamma:[0,1]\to H^{1,\lambda}_0(\Omega) \text{ is a continuous map such that }\gamma(0)=u_\mu \text{ and } \gamma(1)=u_\mu+Tw_\epsilon\}.$$
 Using (\ref{MP1}) and Lemma \ref{MPL}, we obtain
 $$ 0<\frac{\delta^2}{8}\leq c_\mu:=\inf_{\gamma\in\Gamma}\max_{t\in[0,1]} I_\mu(\gamma(t))<\Tilde{c}\textcolor{black}{.}$$
\textcolor{black}{Hence, by the Mountain Pass Theorem, we obtain another critical point $\Tilde{u}$ of the functional $I_\mu$ with energy $I_\mu(\Tilde{u})=c_\mu>0$. Since $c_\mu>0$, one can conclude that $\Tilde{u}\not\equiv 0$.} Thus, $\Tilde{u}\neq u_\mu,$ which contradicts our assumption. Consequently, the problem (\ref{ME1}) admits at least two non-zero, non-negative solutions. In particular, when $n=1,\lambda\geq 1$, then \cite[Theorem 2.6]{KS} guarantees that the solutions are strictly positive.\\
\section*{Acknowledgment:} The authors would like to express their sincere gratitude to Prof. Adimurthi for his invaluable suggestions and support. 

\section*{Declaration of Interest:} The authors declare no competing interests.
\nocite{*}
\bibliographystyle{plain}
\bibliography{References.bib}
\end{document}